\documentclass[11pt]{article}
\usepackage[french, english]{babel}
\usepackage[utf8]{inputenc}
\usepackage{amsmath,amsfonts,amssymb,amsthm,mathrsfs}
\usepackage{minitoc}
\usepackage{bbm}
\renewcommand{\leq}{\leqslant}
\renewcommand{\geq}{\geqslant}
\usepackage[mono=false]{libertine}
\useosf
\usepackage{mathpazo}

\usepackage{wasysym}

\usepackage[dvipsnames]{xcolor}
\usepackage[a4paper,vmargin={3.5cm,3.5cm},hmargin={2cm,2cm}]{geometry}
\usepackage[font=sf, labelfont={sf,bf}, margin=1cm]{caption}
\usepackage{graphicx,graphics}
\usepackage{epsfig}
\usepackage{latexsym}
\usepackage{stmaryrd}
\usepackage{ae,aecompl}

\usepackage[english]{babel}
 \usepackage[colorlinks=true]{hyperref}
\usepackage{pstricks}
\usepackage{enumerate}
\usepackage{tikz}
\usepackage{fontawesome}
\usetikzlibrary{arrows,automata}

\DeclareFixedFont{\beaupetit}{T1}{ftp}{b}{n}{2cm}

\newtheorem{theorem}{Theorem}[]

\newtheorem{proposition}[]{Proposition}

\newtheorem{lemma}[]{Lemma}

\theoremstyle{definition}

\usepackage{todonotes}

\title{\textsc{Parking on supercritical geometric Bienaym\'e--Galton--Watson trees}}
\author{
Linxiao  \textsc{Chen}\thanks{Université Sorbonne Paris-Nord.\hfill  \href{mailto:linxiao.chen@math.univ-paris13.fr}{\texttt{linxiao.chen@math.univ-paris13.fr}}}\qquad\&\qquad
Alice \textsc{Contat}\thanks{Université Sorbonne Paris-Nord \& FSMP.\hfill  \href{mailto:alice.contat@math.cnrs.fr}{\texttt{alice.contat@math.cnrs.fr}}}
}

\date{}
\begin{document}
\maketitle 

\begin{abstract}
Consider a supercritical Bienaym\'e--Galton--Watson tree $ \mathcal{T}$ with geometric offspring distribution. Each vertex of this tree represents a parking spot which can accommodate at most one car. On the top of this tree, we add $(A_u : u \in \mathcal{T})$ i.i.d.\ non negative integers sampled according to a given law $ \mu$, which are the car arrivals on $ \mathcal{T}$. Each car tries to park on its arriving vertex and if the spot is already occupied, it drives towards the root and takes the first available spot. If no spot is found, then it exits the tree without parking.
 In this paper, we provide a criterion to determine the phase of the parking process (subcritical, critical, or supercritical) depending on the generating function of~$ \mu$. \end{abstract}
\section{Introduction}	
The parking process was introduced by Konheim \& Weiss \cite{konheim1966occupancy} in the case of an oriented line. Recently, its generalization on \emph{rooted trees} was initiated by Lackner \& Panholzer \cite{LaP16}. Since then, the phase transition of this model received much attention and was proved and located on \emph{critical} Bienaym\'e--Galton--Watson trees in an increasing level of generality \cite{GP19,curien2022phase,contat2020sharpness,contat2023parking}.

In the case of \emph{supercritical} Bienaym\'e--Galton--Watson trees, the only model for  which the threshold of the transition is explicitly known is a deterministic tree, namely that of the infinite binary tree \cite{aldous2022parking}. The existence of the phase transition is however proved in a more general context \cite{GP19,bahl2021parking} and the model is closely related to Derrida--Retaux model \cite{chen2021derrida}. In this paper, we locate and study the phase transition of the parking process on supercritical Bienaym\'e--Galton--Watson trees with a geometric offspring distribution using essentially enumeration results provided by Chen \cite{chen2021enumeration}.

\paragraph{Parking on trees.}  
We focus here on the case where the underlying tree is a supercritical Bienaym\'e--Galton--Watson tree $ \mathcal{T}$ with a geometric offspring distribution i.e.\ a tree which starts from one individual and where each individual reproduces independently according to the offspring distribution  
$$ \nu := \sum_{k \geq 0} q^k(1-q) \delta_k$$ with parameter $ 1/2<q < 1$, so that $ \nu$ has mean $ q/(1-q)> 1$. We denote by $| \mathcal{T}|$ its size (number of vertices). It is well-known that this tree has a positive probability to be infinite since $q > 1/2$, which is not the case when $ q \leq 1/2$. On the top of this tree, we add $ (A_u : u \in \mathcal{T})$ i.i.d.\ non negative integers with law $ \mu$, which represent the car arrivals on $ \mathcal{T}$. Each car tries to park at its arriving vertex, but if it is already occupied, it drives towards the root until it encounters an empty vertex and parks there. If it arrives at the root without parking, then it exits the root and contributes to the \emph{flux} of outgoing cars. To avoid cases where all cars park on their arrival node, we assume $ \mu(\{0,1\}) <1$. When the tree is finite, the final configuration does not depend on the order in which cars are parked. When the tree is infinite, we circumvent the possible issue by parking the cars layers by layers, see for example \cite[Section 2.1]{aldous2022parking} or the proof of Lemma \ref{lem:finitecluster}.
We introduce
 $$X := \mbox{ the number of cars visiting the root of } \mathcal{T}, $$
so that the flux of outgoing cars equals $ (X-1)_+ = \max (X-1,0)$. 
In \cite[Theorem 3.4]{GP19}, Goldschmidt and Przykucki proved that there are two possible regimes for the parking process on the supercritical Bienaym\'e--Galton--Watson trees $ \mathcal{T}$ depending on the two laws $\mu$ and $\nu$: 
\begin{itemize}
\item Either $ \mathbb{E}[X]< \infty$ \emph{(subcritical regime)}, 
\item Or  $ X = \infty$ as soon as $\mathcal{T}$ is infinite \emph{(supercritical regime)}.
\end{itemize}

The goal of this paper is to provide a criterion to determine if the parking process is subcritical or supercritical depending on the parameter $q$ and the car arrivals distribution $ \mu$.
To state our theorem, we need to introduce $G$ the generating function of the car arrival law $\mu$
$$ G(t) := \sum_{k \geq 0} \mu_k x^k.$$

\begin{theorem} \label{thm:locthereshold}
Suppose that there exists $t_c$ such that 
$$t_c := \inf \{ t> 0, (G(t) - t G'(t))^2 = 2 t^2 G(t) G''(t)\}.$$ 
Then the parking process is subcritical if and only if 
$$ t_c > 1 \quad\mbox{ and } \quad  \frac{t_c G(t_c)}{\varphi(t_c)^2}  \leq q (1-q) ,$$
where $ \varphi(y) = (y+1) G(y) - y(y-1)G'(y)$.
\end{theorem}

Let us make a couple of remarks on this theorem. First to determine if $ t_c > 1$, it suffices to check if $(G(1) - 1 G'(1))^2 > 2 \cdot1^2 G(1) G''(1)$, which is equivalent to $ 2 \sigma^2 + m^2 < 1$ where $m$ and $ \sigma^2$ are respectively the mean and the variance of $ \mu$. Moreover,  the second inequality is always satisfied for $ q = 1/2$. In particular, when $q \searrow 1/2$, we recover the same criterion for the phase transition as that proved by Curien and Hénard in \cite[Theorem 1]{curien2022phase}. However, the nature of the phase transition is different when $q=1/2$, since the tree $  \mathcal{T}$ is almost surely finite. In this case, the flux of outgoing cars has an infinite expectation in the supercritical and critical phase of the parking process, whereas this expectation is finite in the subcritical phase, see also  \cite{bahl2021parking}.

Moreover, when $ t_c \leq 1$, the parking process is always supercritical. As an example, if $ \mu$ is a car arrival distribution such that the radius of convergence of $G$ equals $1$, then for all $ \alpha \in (0,1]$, the parking process with car arrival distribution $ \mu_ \alpha = (1- \alpha) \delta_0 + \alpha \mu$ is supercritical.  In other words, the phase transition is trivial for the stochastically increasing family of car arrival distributions $( \mu_ \alpha : \alpha \in [0,1])$ since it occurs at $ \alpha_c := 0$: if $\alpha = 0$, the parking process is subcritical (no cars arriving) and if $\alpha \geq 0$, then the parking process is supercritical.

When $ t_c$ does not exist, we are in the \emph{dense phase} in the classification of Chen \cite{chen2021enumeration}. Our criterion remains valid by verifying the two inequalities for $ \widetilde{t}_c$ the radius of convergence of $G$. %

The proof follows essentially the same ideas as in \cite{aldous2022parking} and rely on a combinatorial decomposition of the final configuration of parking into clusters of parked cars and on the enumeration of plane \emph{fully parked trees} by Chen \cite{chen2021enumeration}. We use the fact that the flux of outgoing cars is almost surely finite in the subcritical regime to obtain an equation involving the generating function of fully parked trees which characterizes the subcritical regime, see Proposition \ref{prop:subcriticF}.

\paragraph{Examples.}  
Let us give some examples of explicit laws where the criterion to determine the threshold of the parking process is explicit. To make this criterion more intelligible, it is more convenient to fix the car arrival law and determine the range of parameter $q$ for which the parking process is subcritical. We will come back to this examples in Section \ref{sec:examples}. \medskip \\ 
\textit{Binary arrivals.} The simplest example of non-trivial car arrival distribution is the binary distribution $ \mu_ \alpha$: each vertex receives independently either two cars with probability $ \alpha/2$ or no car with probability $ 1 - \alpha/2$  for some parameter $ \alpha \in [0,1]$.  In this case, the generating function of  $ \mu_ \alpha$ is $ G(t) := (1- \alpha/2) + \alpha/2 t^2$. Then, the parking process is subcritical if and only if 
$$  \alpha \leq 2 - \sqrt{3} \quad \mbox{ and } \quad q \leq \frac{1}{2} \left( 1 + \sqrt{1 - \frac{6 \sqrt{2 \sqrt{3}+3}}{ \sqrt{ \alpha (2- \alpha)} \left( 3 + \sqrt{ 2 \sqrt{3}- 3} \sqrt{ \frac{2- \alpha}{ \alpha}}\right)^2}}\right). $$\medskip \\
\textit{Geometric arrivals.} If $ \mu_ \alpha = \sum_{k \geq 0} p^k (1-p) \delta_k$ is a geometric distribution with mean $ \alpha = p/(1-p)$, then the parking process is subcritical if and only if  
$$  \alpha \leq \frac{1}{3}  \quad \mbox{ and } \quad  q \leq \frac{1}{2} \left( 1 + \frac{(1-3 \alpha )^{3/2}}{1 + 9 \alpha}\right). $$ \medskip \\
\textit{Poisson arrivals.} Lastly, if $ \mu_ \alpha$ is Poisson with mean $ \alpha$, then the parking process is subcritical if and only if 
$$ \alpha \leq \sqrt{2} - 1  \quad \mbox{ and } q \leq \frac{1}{2} \left( 1+ \sqrt{ 1 - \frac{ 2( \sqrt{2}- 1)  \alpha \mathrm{e}^{ \alpha -( \sqrt{2}- 1)}}{ ( \alpha + 3  - 2 \sqrt{2})^2}}\right).$$
\textit{Non generic dilute distributions.} Another class of examples is when the car arrivals distribution is non generic. More precisely, let us take the generating function $G$ of the form 
$$ G(t) = P(t) + C \left(1- \frac{t}{ \rho} \right)^ \alpha.$$
where $P$ is a polynomial, and we fix  a negative constant $C < 0$ ,  a radius of convergence $ \rho > 1$ and an exponent $\alpha \in (2,3)$. We assume that $G$ is the generating function of a critical law, in the sense that $ t_c = \rho$ and the parameter $q$ of the tree $ \mathcal{T}$ is such that the inequality on the right in Theorem \ref{thm:locthereshold} is an equality. 
Then, the generating function of $X$ the number of cars visiting the root has a radius of convergence $ t_c$. In particular, the tail of the law of $X$ decays exponentially fast. However, this is not the case for the law of the size $ | \mathcal{C} ( \varnothing)|$ of the cluster of the root. For example, using the asymptotic provided by Chen in  \cite[Theorem 1]{chen2021enumeration}, we can show that conditionally on $X=1$, the probability that $ | \mathcal{C} ( \varnothing |)$ has size $n$ decays polynomially:
$$ \mathbb{P} \left( | \mathcal{C} ( \varnothing)| = n | X=1 \right)  \underset{n \to \infty}{\sim} \mathrm{cst} \cdot n^{\frac{- 2 \alpha +1}{ \alpha-1}}, $$
for some constant $\mathrm{cst}> 0$ that depends on $ \alpha$, see Section \ref{sec:examples}. Note that the exponent  $(- 2 \alpha +1)/ ( \alpha-1)$ is increasing from  $ -3$ to $ -5/2$ for $ \alpha \in (2,3)$.

\paragraph{Outline of the paper.} In Section \ref{sec:trans}, we show the dichotomy between the two regimes for the parking process, as well as our main decomposition which allows us to characterizes these two regimes using the generating function of \emph{fully parked trees}. Section \ref{sec:proba} is devoted to the proof of our main Theorem \ref{thm:locthereshold}. We also detail in this section some examples to which  our theorem applies and the critical exponents which appear in this examples. 

\paragraph{Acknowledgments.} We thank Nicolas Curien for stimulating discussions. The second author acknowledges partial support from SuPerGRandMa, the ERC CoG 101087572.

\section{Phase transition}\label{sec:trans}
In this section, we prove some ``probabilistic" properties of the parking process that characterize the two different regimes of this model. 
\subsection{Two regimes}
We start by proving the dichotomy between the two different regimes of the parking process. This distinction already appears in  \cite[Theorem 3.4]{GP19} but in our case, we show here a slightly more narrow property to characterize the subcritical regime. 
\begin{lemma}[Dichotomy subcritical/supercritical regime]\label{lem:dicho} We have the following dichotomy : \\
\textbf{Subcritical regime.}  Either $ \left( \frac{q}{1-q}\right)^n \mathbb{P} \left( X \geq n \right)$ is bounded,  \\
\textbf{Supercritical regime.}  Or $ X = \infty$ as long as $ \mathcal{T}$ is infinite. 
\end{lemma}
Notice in particular that the number of cars visiting the root is almost surely finite and has a finite expectation in the subcritical regime. 

\proof Let us denote by $Z_n$ the number of vertices at level $n$ (at distance $n$ from the root) in the tree $ \mathcal{T}$. Since the offspring distribution $ \nu$ of $ \mathcal{T}$ has mean $ q /(1-q)$ and has bounded variance, then  $ (q /(1-q))^{-n}Z_n$ is a positive martingale which is bounded in $L^2$. Thus, it converges is $L^2$ and the probability that the limit equals $0$ is the same as the probability that the tree $ \mathcal{T}$ is finite. As a consequence, we have

$$\lim_{ \varepsilon \to 0} \lim_{n \to \infty}  \mathbb{P} \left( Z_n \geq \varepsilon  \left( \frac{q}{1-q}\right)^n \right) = \mathbb{P} \left( |\mathcal{T} | = \infty\right).  $$ 
Assume now  that the sequence $(q /(1-q))^n \mathbb{P} \left( X \geq n \right)$ is not bounded. Then this is also the case for the sequence $(q /(1-q))^n \mathbb{P} \left( X \geq n +k \right)$ for any integer $k$. 
Then on each of the $ Z_n$ vertices at level $n$, independent  copies of Bienaym\'e--Galton--Watson trees are attached and for each vertex $u$ at height $n$, we can consider $ X(u)$ the random number of cars visiting the vertex $u$. We obtain the upper bound 

\begin{eqnarray*} \mathbb{P} \left( \bigcap_{u \text{ at height } n} \left\{ X(u) \leq n + k \right\}  \mbox{ and } Z_n \geq \varepsilon \left( \frac{q}{1-q}\right)^n\right)  &\leq& (1- \mathbb{P} \left( X > n+k\right))^{ \varepsilon \left( \frac{ q}{1-q}\right)^n} \\
 &\leq&  \exp\left( - \varepsilon\left( \frac{q}{1-q}\right)^n \mathbb{P} \left( X > n+k\right) \right). 
\end{eqnarray*}
and the right-hand side goes to $0$ along a subsequence. But on the complement of the event on the left-hand side, either $ X \geq k$ or $  Z_n \leq \varepsilon  ({q}/{(1-q)})^n$. Thus, with probability at least $ \mathbb{P} \left( |\mathcal{T} | = \infty\right)$, we have $ X \geq k$. Since $k$ is arbitrary, this implies that $X = \infty$ with probability at least $ \mathbb{P} \left( |\mathcal{T} | = \infty\right)$. Moreover, the number of cars visiting the root can not be infinite when the tree is infinite, thus  $X = \infty$ with probability at most $ \mathbb{P} \left( |\mathcal{T} | = \infty\right)$, which concludes the proof. 

 \endproof

Note that $X$ the number of cars visiting the root is at least the number of cars arriving at the root. In particular, if the radius of convergence of $G$ the generating function of the car arrivals is strictly smaller than $ q/(1-q)$, then it is also the case for the generating function of the variable $X$, which implies that the parking process is supercritical. 

We are now interested in the possible sizes of the clusters that we can observe depending on the regime. We call \emph{black} the clusters of parked vertices and \emph{white} the clusters of empty vertices. 
We now prove that all black clusters are finite in the subcritical regime (which is not true for all white clusters). This is the analog of \cite[Proposition 1]{aldous2022parking} in our case and the proof essentially follows the same line. We adapt it in our case to keep our paper  self-contained. 
\begin{lemma}\label{lem:finitecluster} In the subcritical regime, there is no infinite black cluster.
\end{lemma}

\proof This proof is mostly an adaptation of the proof of \cite[Proposition 1]{aldous2022parking}. As mentioned in \cite[Section 2.1]{aldous2022parking}, we park the lowest cars first to avoid Abelianity issues when the tree $ \mathcal{T}$ is infinite. More precisely, for each $n \geq 0$, we denote by $ [ \mathcal{T}]_n$  the finite tree obtained by cutting $ \mathcal{T}$ above height $n$, so that $ [ \mathcal{T}]_n$ is made of all vertices of $   \mathcal{T}$ of height smaller or equal than $n$. Then, parking the cars arriving on $ [ \mathcal{T}]_n$, we can define the variables 
$$ X_n (u) \quad u \in  [ \mathcal{T}]_n $$
denoting the number of cars visiting the vertex $u \in  [ \mathcal{T}]_n $ in the final configuration of parking in $ [ \mathcal{T}]_n $. Since for all $u \in  [ \mathcal{T}]_n $, the sequence $ (X_n (u) : n \geq 0)$ is non-decreasing, we can define 
$$ X (u) := \lim_{n \to \infty} X_n(u)$$ which represents the number of cars visiting $u \in \mathcal{T}$ even when $ \mathcal{T}$ is infinite.  

Now the parking rules are clear even on infinite trees, we assume the tree $ \mathcal{T}$ is infinite (otherwise they can not be infinite clusters) and the law $ \mu $ is subcritical. Thus, all variables $ (X(u) : u \in \mathcal{T})$ are almost surely finite and we denote by $ \mathcal{C} ( \varnothing)$ the cluster of the root of $ \mathcal{T}$. Let us fix $ p \geq 0$ and consider $ \mathcal{E} = \{ X ( \varnothing ) = p \mbox{ and } \mathcal{C} ( \varnothing ) \mbox{ is infinite}\}$. We can construct a sequence of stopping times $ \theta_1 < \theta_1 < \dots$ such that $ \theta_1 := \inf \{ n \geq 0 : X_n ( \varnothing) = p\}$ and for all $ i \geq 1$, we define sequentially $ \theta_{i+1} := \infty \{ n > \theta_i :  \varnothing \leftrightarrow \partial [ \mathcal{T}]_n \}$ meaning the $ \theta_{i+1}$ is the first $n > \theta_i$ such that there exists a path connecting the root $ \varnothing$ and the level $n$ in $ [ \mathcal{T}]_n$ which satisfies $ X_n (u) \geq 1$ for all vertices $u$ in this path. Let us denote by $ v_{ i}$ the leftmost vertex at level $ \theta_i$ which is connected to the root by such a path (if there is one) for $i \geq 2$. Note that on the event $ \mathcal{E}$, then all stopping times $ \theta_i$ are finite and thus, for all $i \geq 2$, the vertex $v_i$ is well-defined and its children have at most $1$ car arriving on it. 
In particular, we obtain for all $n \geq 2$
\begin{eqnarray*}
 \mathbb{P} \left( \mathcal{E}\right) &\leq& \prod_{i=2}^{n} \mathbb{P} \left( \mbox{all children of } v_i \mbox{ receive at most one car} | \theta_i < \infty \right) \\
 &\leq& \left( \sum_{ k\geq 0} q^k ( 1-q) ( \mu_0 + \mu_1)^k \right)^{n-1}= \left( \frac{1-q}{1-q (\mu_0 + \mu_1)} \right)^{n-1}.
\end{eqnarray*}
Since we assumed $ \mu(\{0,1\}) < 1$ and $n$ is arbitrary, the fraction this implies $ \mathbb{P} \left( \mathcal{E}\right) = 0$ and this is true for all value of $p$, which concludes the proof. 

\endproof
\subsection{Fully parked trees}
This section presents a description of the fully parked components and  mainly contains the results from Section 2 in \cite{chen2021enumeration}.

Recall our notation $ (a_x : x \in \mathbf{t})$ for the car arrivals on the vertices of a tree $ \mathbf{t}$.  A \emph{fully parked tree} is a tree decorated with its car arrivals configuration such that in the final configuration, all vertices are occupied spots and  they may be outgoing cars. We then denote by $ \mathbb{T}_n^{(k)}$ the set of (plane) fully parked trees with $n$ vertices and $k$ outgoing cars (hence $ n+k$ arriving cars) and introduce the bivariate generating function 
\begin{equation*} F (x,y) = \sum_{n \geq1} \sum_{k \geq 0} \sum_{ \mathbf{t} \in \mathbb{T}_n^{(k)}} w( \mathbf{t}) x^n y^k \quad \mbox{ where} \quad w( \mathbf{t}):= \prod_{ x \in \mathbf{t}} \mu_ {a_x}.\end{equation*}

The first important tool is a recursive equation obtained from an ``à la Tutte" recursive decomposition, i.e. by decomposing each fully parked tree according to the splitting at its root. Indeed, if the root of $ \mathbf{t} \in \mathbb{T}_n^{(k)}$ has $j \geq 0$ children, then all these $j$ children are the respective roots of a sequence $( \mathbf{t}_i : 1 \leq i \leq j)$ of  fully parked trees. Moreover, if for all $1 \leq i \leq j$, the tree $ \mathbf{t}_i$ belongs to $\mathbb{T}_{n_i}^{(k_i)}$  then the fact that  $\mathbf{t}$ has $n$ vertices and $k$ outgoing cars implies that $n = 1 + \sum_{i=1}^{j} n_i$ and $ k = \sum_{i = 1}^{j} k_i + a_{ \varnothing} -1$ where $ a_{ \varnothing}$ is the number of cars arriving at the root of $ \mathbf{t}$, see Figure \ref{fig:tutte}. Note that since the root of $ \mathbf{t}$ is an occupied spot, then we can not have $a_{ \varnothing} = 0$ and $ k_i = 0$ for all $ 1 \leq i \leq j$ simultaneously. Translating this observation on $F$ gives the following equation (see also \cite[Equation 21]{chen2021enumeration})
\begin{equation}\label{eq:tutte} F(x,y) = \frac{x}{y} \left( \frac{G(y)}{1 - F(x,y)} - \frac{G(0)}{1 - F(x,0)}\right).
\end{equation}

\begin{figure}[!h]
 \begin{center}
 \includegraphics[width=15cm]{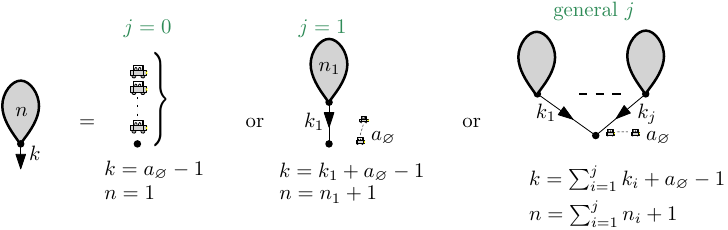}
 \caption{\label{fig:tutte}Illustration of the ``a la Tutte" recursive decomposition. On the left, a fully parked tree of size $n$ and $p$ outgoing cars. On the right, the different possibility depending on the number $j$ of the children of the root vertex.}
 \end{center}
 \end{figure}
Note that this equation characterizes $F$ since the coefficient in front of $x^n$ on the left-hand side only depends on the coefficients of $F(x,y)$ in front of $ x^{k}$ for $0 \leq k \leq n-1$. It involves not only $F(x,y)$ but also $F(x,0)$ its specialization at $y=0$. We thus apply the usual \emph{kernel method} \cite{MBM:quadratic} to solve this equation as it is done in \cite[Section 2]{chen2021enumeration}. We introduce $ P(f,f_0,x,y) =  \frac{G(y)}{1 - f} -  \frac{G(0)}{1 - f_0} + \frac{x f}{y}$, so that Equation~\eqref{eq:tutte} is equivalent to $ P(F(x,y), F(x,0), x,y) = 0$. The key idea is to find a function $ Y = Y(x)$ such that $ \partial_f P(F(x,Y(x)), F(x,0), x,Y(x)) = 0$ so that differentiating $ P(F(x,Y(x)), F(x,0), x,Y(x)) = 0$ with respect to $y$, we also obtain that $ \partial_y P(F(x,Y(x)), F(x,0), x,Y(x)) = 0$. Using the explicit form for the function $P$, we then have 
\begin{equation} \frac{G(Y)}{1- F(x,Y)} = \frac{Y}{x}, \qquad  \frac{G'(Y)}{1- F(x,Y)} = \frac{F(x,Y)}{x}, \qquad \mbox{and} \qquad  \frac{G(Y)}{1- F(x,Y)}  - \frac{G(0)}{1- F(x,0)}  = \frac{Y F(x,Y)}{x}.
\end{equation}
The existence of  a solution for $Y$ is guaranteed by \cite{MBM:quadratic} and eliminating the variable $F(x,Y)$, we can express $x$ and $F (x,0)$ in terms of $Y$ and get
\begin{equation}\label{eq:paramf0} x = \frac{Y G(Y)}{(G(Y)+ Y G'(Y))^2} \qquad \mbox{and} \qquad F(x,0) = 1 -  \frac{G(0) G(Y)}{(G(Y)- Y G'(Y))^2}.
\end{equation}
Plugging this back into Tutte's equation \eqref{eq:tutte}, we can obtain a parametrization of $ F(x,y)$. In particular, we can obtain its value  and radius of convergence in $x$ at $y=1$ , which will be useful thereafter and which we sum up in the following proposition. See also \cite[Proposition 1]{chen2021enumeration}.

\begin{proposition}\label{prop:parametization} 
As in Theorem \ref{thm:locthereshold}, we assume that there exists $t_c$ such that 
$$t_c := \inf \{ t> 0, (G(t) - t G'(t))^2 = 2 t^2 G(t) G''(t)\}.$$ 
We also suppose that $t_c>1$. Then, writing
$x = \hat{x} (Y)  = \frac{Y G(Y)}{(G(Y)+ Y G'(Y))^2},$ we have  \begin{eqnarray*} F(x,1) = \hat{F} (Y,1) =  \left\{ \begin{array}{lll}1+  \frac{ \varphi( Y) }{2 (G(Y)+ Y G'(Y))} \left( - \sqrt{1 - \frac{4 Y G(Y)}{\varphi (Y)^2}} - 1 \right) &  \text{ when } &0 \leq Y \leq 1, \\
1+  \frac{ \varphi( Y) }{2 (G(Y)+ Y G'(Y))} \left( \sqrt{1 - \frac{4 Y G(Y)}{\varphi (Y)^2}} - 1 \right) &  \text{ when } &1 < Y \leq t_c, \end{array}\right. \end{eqnarray*}
where $ \varphi(y) = (y+1) G(y) - y(y-1)G'(y)$. Moreover, the radius of convergence of $x \mapsto F(x,1)$ is $ \hat{x} (t_c) = \frac{ t_c G(t_c) }{(G(t_c)+ t_c G'(t_c))^2} $. 
\end{proposition}
 
\proof The parametrization follows easily from the kernel method applied to Tutte's equation \eqref{eq:tutte}. In particular, it is well-defined and the quantity inside the square root is non negative for all $ 0 \leq Y \leq t_c$. To see that $\hat{x}(t_c)$ is the radius of convergence of $F(x,1)$, we first observe that the parametrization implies that this radius is at least  $\hat{x}(t_c)$, see \cite{MBM:quadratic}. If $G$ or one of its first two derivative of $G$ is infinite at $ t_c$, then the result is clearly true. Otherwise, thanks to our parametrization, we can show that the function $ F(x,1) $ and its first derivative have a finite limit as $ x \to  \hat{x}(t_c)$, see also \cite[Lemma 14]{chen2021enumeration}. Moreover, by tedious computations, we can show that 
\begin{equation*}
 \lim_{x \to  \hat{x}(t_c)}\frac{ \partial^2 F(x,1)}{ \partial x^2} = \lim_{ y \to t_c} \frac{(G(y)+yG'(y))^5}{y G(y)} \left( \sqrt{ 1 + \frac{4 y G(y)}{ \varphi (y)^2 - 4 yG(y)}} - 1\right) \cdot \frac{1}{(G(y) - y G'(y))^2 - 2 y^2 G(y) G''(y)}
 \end{equation*}
 The first factor has a finite limit whereas the denominator vanishes on the right-hand side. This concludes the proof.

 \endproof

\subsection{Decomposition into fully-parked components}
In this section, we present the main decomposition from which we derive our characterization of the subcritical regime for the parking process.
We  remember that $X$ is the number of cars visiting the root  of $ \mathcal{T}$ with $ p_ \circ = \mathbb{P} \left( X=0\right)$ and we write $ \mathcal{C}( \varnothing)$ for the cluster of the root  (either white (void) or black (parked)) in $ \mathcal{T}$ after parking.

Let us fix $k \geq 0$. For all $x \geq 0$, we denote by $[y^k] F(x,y)$ denote the coefficient in front of $y^k$ in the (formal) power series $ y \mapsto F(x,y)$. For all $x \geq 0$ such that  $[y^k] F(x,y) < \infty$, 
we say that $ \mathcal{F}$ is a (random) fully parked tree with $k$ outgoing cars and $x$-Boltzmann law if for all $n \geq 0$ and $ \mathbf{t} \in \mathbb{T}_n^{(k)}$, we have 
\[ \mathbb{P} \left( \mathcal{F} = \mathbf{t}\right) = \frac{1}{[y^k] F(x,y)} x^n w( \mathbf{t}). \]
In the subcritical regime, we can link the variable $X$ and the coefficients of $F$ through the following lemma. 
\begin{lemma}\label{lem:xboltzmann} When the law $ \mu$ is subcritical for the parking process,  for all $ k \geq 0$, we have 
\begin{equation*} \mathbb{P} \left( X= k+1 \right) = \frac{1- q p_ \circ}{q} [y^k] F \left( \frac{ q (1-q)}{( 1- q p_ \circ)^2}, y \right). \end{equation*}
Moreover, conditionally on $ X = k+1$, the cluster of the root $  \mathcal{C}( \varnothing)$ is a  fully parked tree with $k$ outgoing cars and $ \frac{ q (1-q)}{( 1- q p_ \circ)^2}$-Bolztmann law. 
\end{lemma}
 \begin{proof}
The proof follows from a simple decomposition according to the cluster of the root. We emphasize here the important role of the geometric offspring distribution here: in the tree $ \mathcal{T}$, each vertex has a probability $ q^k (1-q)$ to have $k$ children, which is equivalent to saying that each edge carries a weight $q$ and each vertex carries a weight $(1-q)$. Let us fix  a finite fully parked tree $\mathbf{t} \in  \mathbb{T}_n^{(k)}$ and compute the probability that $   \mathcal{C}( \varnothing) =  \mathbf{t}$. The weight of this tree is the product of the weight $ w ( \mathbf{t})$ of the car arrivals on it, the weight $(1-q)^n$ of its $n$ vertices and $q^{n-1}$ for its $n-1$ edges. Moreover, since $\mathbf{t}$ has  $n$ vertices, there are $2n-1$ corners where there can be empty spots. Each empty spot comes with a weight  $ q$ for edge between the vertex and a vertex of $ \mathbf{t}$ multiplied by  $p_ \circ$ since it can be seen as the root of a Bienaym\'e--Galton--Watson with an empty root. The weight of each corner is then $ \frac{1}{1 - q p_ \circ}$. See Figure  \ref{fig:decomp} for an example. We thus obtain 
\begin{equation*} \mathbb{P} \left(   \mathcal{C}( \varnothing) =  \mathbf{t}\right) = w ( \mathbf{t}) (1-q)^n q^{n-1} \left( \frac{1}{1 - q p_ \circ}\right)^{2n-1} = \frac{(1-q p_ \circ)}{q} \cdot \left( \frac{q (1-q)}{ (1 - q p_ \circ)^2}\right)^n w ( \mathbf{t}) \end{equation*}
 
 \begin{figure}[!h]
  \begin{center}
  \includegraphics[width=16cm]{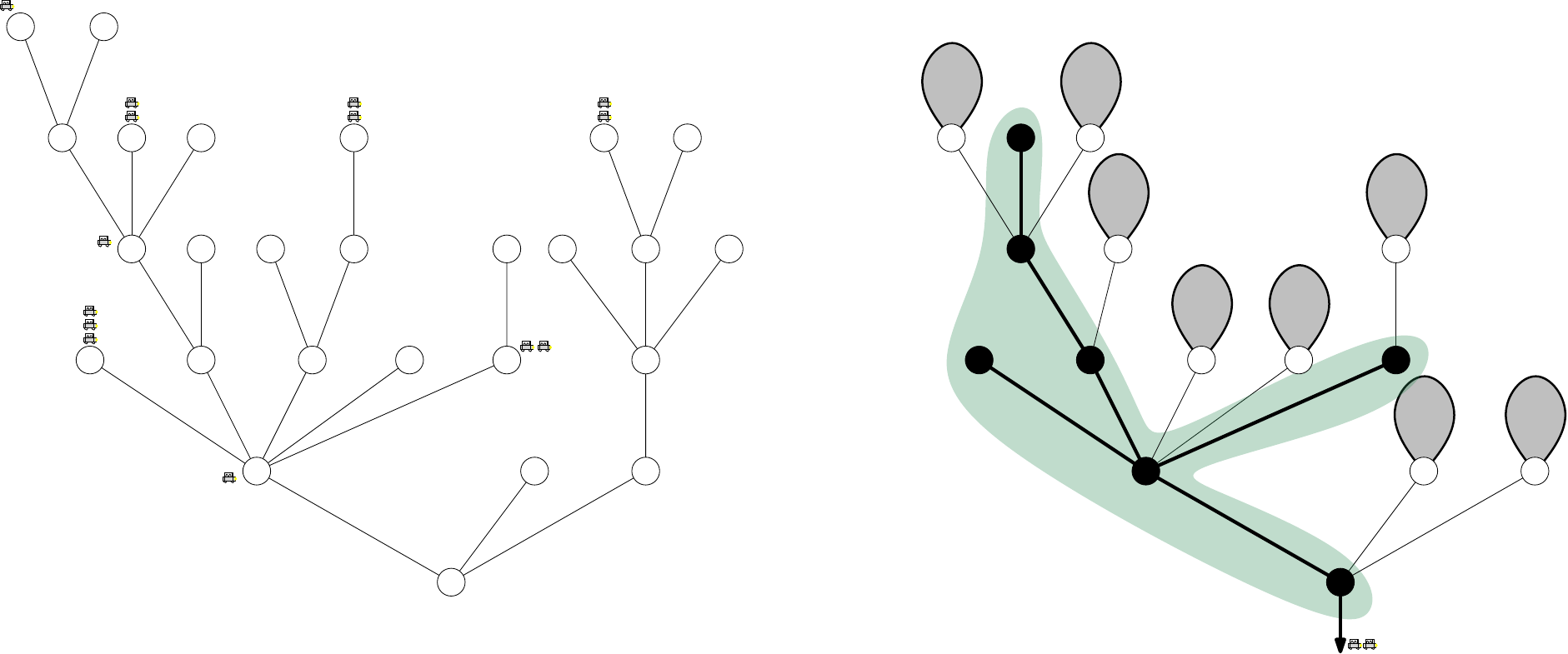}  \caption{ \label{fig:decomp}On the left, a tree together with its car arrivals configuration. On the right, the (black) cluster of root of this tree together with all empty spots attached to it in the initial tree. All this empty spots are the root of independent Bienaym\'e--Galton--Watson trees where in the final configuration, the root is empty. This cluster is a fully parked tree with $7$ vertices and there are $8$ empty spots attached to it. The probability to observe this cluster is $( \mu_0^2 \mu_1^2 \mu_2^2 \mu_3) q^6 (1-q)^7 (q p_ \circ)^8$ }
  \end{center}
  \end{figure}
Moreover, we know from Lemma \ref{lem:finitecluster} that all black clusters are finite in the subcritical regime. Thus, for all $ k \geq 0$
\begin{eqnarray*} \mathbb{P} \left( X= k+1 \right) &=& \mathbb{P} \left( X= k+1 \mbox{ and } | \mathcal{C} ( \varnothing)| < \infty \right) \\
&=& \sum_{n \geq 1}\sum_{ \mathbf{t} \in \mathbb{T}_{n}^{(k)}} \mathbb{P} \left(  | \mathcal{C} ( \varnothing)|  = \mathbf{t}\right) \\
&=& \sum_{n \geq 1}\sum_{ \mathbf{t} \in \mathbb{T}_{n}^{(k)}}\frac{(1-q p_ \circ)}{q} \cdot \left( \frac{q (1-q)}{ (1 - q p_ \circ)^2}\right)^n w ( \mathbf{t})  = \frac{(1- q p_ \circ)}{q} [y^k] F \left( \frac{ q (1-q)}{( 1- q p_ \circ)^2}, y \right),
\end{eqnarray*}
which concludes the proof. 
 \end{proof}

Moreover, in the subcritical regime, the variable $X$ is almost surely finite, thus splitting according to the values of $X$, we obtain
\begin{equation}\label{eq:recflux} 1 = p_ \circ + \sum_{k \geq 0} \mathbb{P} (X = k+1) = p_\circ + \frac{1- q p_ \circ}{q} F \left( \frac{ q (1-q)}{( 1- q p_ \circ)^2}, 1 \right).
\end{equation} 
Note that the function $p \mapsto p + \frac{1- q p }{q} F \left( \frac{ q (1-q)}{( 1- q p)^2}, 1 \right)$ is an increasing function of $p$. Thus the solution to this equation is unique if there is one. 
This gives us the following characterization of the subcritical regime.

\begin{proposition}[F-characterization of the subcritical regime]\label{prop:subcriticF} The law $ \mu$ is subcritical for the parking process on a Bienaym\'e--Galton--Watson tree with geometric offspring distribution with parameter $ q$ if and only if there exists a positive solution $p_{ \circ}>0$ to the equation 
\begin{equation}\label{eq:characp} \frac{1 - q p}{q} \cdot F\left( \frac{q (1-q)}{(1-q p)^2}, 1\right) + p = 1.
\end{equation}
\end{proposition}
 The proof follows the same line as \cite[Proposition 2]{aldous2022parking}
\proof
If  the parking process is subcritical, then $p_ \circ$ is solution of Equation \eqref{eq:characp} by the above computations. 
Reciprocally suppose that there exists a solution $ p_ \circ$ to the equation \eqref{eq:characp}. Then there exist a random variable $Z$ whose generating function is 
\begin{equation*} \mathbf{Z} (y) := \frac{1 - q p_ \circ}{q} \cdot F\left( \frac{q (1-q)}{(1-q p_ \circ)^2}, y\right) + p_ \circ. 
\end{equation*}
Then, writing $ x_ \circ :=q (1-q)/(1-q p_ \circ)^2$ and $F_0 (x):= F(x,0)$,  we compute
\begin{align}
&\hspace{-1cm}\frac{1}{y} \left(\frac{(1-q) G(y)}{\left( 1- q \mathbf{Z}(y)\right) }- \frac{(1-q) G(0)}{\left( 1- q \mathbf{Z}(0) \right) } \right) + \frac{(1-q) G(0)}{\left( 1- q  \mathbf{Z}(0) \right)}  \nonumber \\
=&\frac{1}{y} \left(\frac{(1-q) G(y)}{(1- q p_ \circ) \left( 1 -  F\left( x_ \circ, y\right)\right)}- \frac{(1-q) G(0)}{(1- q p_ \circ) \left( 1 -  F\left( x_ \circ, 0\right)\right)} \right) + \frac{(1-q) G(0)}{(1- q p_ \circ) \left( 1 -  F\left( x_ \circ, 0\right)\right)} \nonumber \\
\overset{ \eqref{eq:tutte}}{=}& \frac{1-q}{1-q p_ \circ}\cdot \frac{1}{x_ \circ} F(x_ \circ, y) +  \frac{(1-q) G(0)}{(1- q p_ \circ) \left( 1 -  F\left( x_ \circ, 0\right)\right)} \label{eq:recgen}
\end{align}
Moreover, from Equation \eqref{eq:characp}, we obtain 
$$F \left( x_ \circ, 1 \right) = \frac{(1-p_ \circ)q}{1- p_ \circ q} \quad \mbox{and} \quad 1- F \left( x_ \circ, 1 \right) = \frac{ 1-q}{1 - p_ \circ q}.$$
Thus, using Tutte's equation \eqref{eq:tutte} with $y=1$, we obtain
\begin{eqnarray*}
 \frac{(1-q)G(0)}{(1-q p_ \circ) \left(1- F \left( x_ \circ, 0 \right)\right)} &=& \frac{1-q}{1-q p_ \circ} \left( \frac{1}{1- F \left( x_ \circ, 1 \right)}- \frac{F( x_ \circ,1)}{x_ \circ}\right) \\
&=& 1  - \frac{(1-q)}{(1-q p_ \circ)} \cdot \frac{(1-q p_ \circ)^2}{q(1-q)} \cdot \frac{(1-p_ \circ)q}{(1-p_ \circ q)}\\
&=&1 - (1- p_ \circ) = p_ \circ.
\end{eqnarray*}
Hence, plugging this equality in \eqref{eq:recgen}, we obtain 
$$ \mathbf{Z} (y) = \frac{1}{y} \left(\frac{(1-q) G(y)}{\left( 1- q \mathbf{Z}(y)\right) }- \frac{(1-q) G(0)}{\left( 1- q \mathbf{Z}(0) \right) } \right) + \frac{(1-q) G(0)}{\left( 1- q  \mathbf{Z}(0) \right)}. $$
This identity is equivalent to the following recursive distributional equation for $Z$:
$$ Z \overset{(d)}{=} \left(\sum_{i = 1}^{Y} Z_i + A -1\right)_+, $$
where on the right-hand side, the variables $(Z_i : i \geq 1)$ are i.i.d. copies of $Z$, the variable $Y$ has law $\nu$ the offspring distribution of the tree, the variable $A$ has law $ \mu$ the car arrivals distribution and all variables are independent.

This recursive equation enables us to construct the tree $ \mathcal{T}$ (using the variable $ Y$) together with its car arrivals (using the variable $ A$) layers by layers in a coherent matter ans so that in the end, the flux of out cars has law $Z$. 
In particular, it implies that the parking process with car arrival law $ \mu$ on the tree $ \mathcal{T}$ is subcritical. 
 \endproof

\section{Probabilistic consequences}\label{sec:proba}

\subsection{Location of the thereshold: Theorem \ref{thm:locthereshold}}

We have now all the tools to prove Theorem \ref{thm:locthereshold}.  Recall that thanks to Proposition \ref{prop:subcriticF}, the parking process is subcritical if and only if \eqref{eq:characp} has a positive solution $ p_ \circ$. Let us write $x_ \circ = q(1-q)/(1- q p_ \circ)^2$, Equation \eqref{eq:characp} is equivalent to 
\begin{equation}\label{eq:characbis}\sqrt{\frac{1-q}{qx_ \circ}} (F(x_ \circ,1) - 1) = - \frac{1-q}{q}.\end{equation}
Note that the function $ x \mapsto \sqrt{(1-q)/(qx)} (F(x,1) - 1)$ is increasing in $x$ and assuming $ t_c \geq 1$, let us compute its value at $ \hat{x}(1)$ using the parametrization given in Proposition \ref{prop:parametization}. Recalling  $ \varphi(y) = (y+1) G(y) - y(y-1)G'(y)$, we can compute 
\begin{eqnarray*}
\sqrt{(1-q)/(q \hat{x}(1))} (F(\hat{x}(1),1) - 1) &=&  \sqrt{\frac{(1-q)(G(1)+ 1 G'(1))^2}{q  G(1)}} \cdot \frac{ - \varphi( 1) }{2 (G(1)+ 1 G'(1))} \\
&=& - \sqrt{\frac{1-q}{q}}  \leq - \frac{1-q}{q}
\end{eqnarray*}
 for $ q \in (1/2,1)$. Thus, is $ x_ \circ$ is a solution of Equation \eqref{eq:characbis}, then $ x_ \circ$ can be written as $ \hat{x} (Y_ \circ)$ for some $Y_ \circ \geq 1$. In particular, the parking process is always supercritical if $ t_c < 1$. Suppose now $ t_c \geq 1$. It suffices to check if  $ \sqrt{(1-q)/(qx_c)} (F(x_c,1) - 1) \geq - (1-q)/q$ where $x_c$ is the radius of convergence of $ x \mapsto F(x,1)$. Moreover, again by Proposition \ref{prop:parametization}, this radius is $ \hat{x}(t_c)$ and
we obtain
\begin{equation*}  \sqrt{\frac{(1-q) \varphi(t_c)^2}{q \cdot 4 t_c G(t_c)}}  \left( \sqrt{1 - \frac{4 t_c G(t_c)}{\varphi (t_c)^2}} - 1 \right)\geq - \frac{1-q}{q}\quad \Leftrightarrow\quad \sqrt{1 - \frac{4 t_c G(t_c)}{\varphi (t_c)^2}} \geq 1 -   \sqrt{\frac{1-q}{q}} \sqrt{\frac{4 t_c G(t_c)}{\varphi (t_c)^2}}.
\end{equation*}
Since both sides of the inequality on the right are positive (this is a consequence of the parametrization for the right-hand-side)  and we can take the square of both sides and substracting $1$, we obtain 
$$- \frac{4 t_c G(t_c)}{\varphi (t_c)^2} \geq  \frac{1-q}{q} \cdot  \frac{4 t_c G(t_c)}{\varphi (t_c)^2} - 2 \sqrt{ \frac{1-q}{q}} \sqrt{\frac{4 t_c G(t_c)}{\varphi (t_c)^2}}  $$
Finally, we obtain that the existence of a positive solution to  Equation \eqref{eq:characp} is equivalent to 
$$  \frac{ t_c G(t_c)}{\varphi (t_c)^2} \leq  q (1-q),$$
and Theorem \ref{thm:locthereshold} follows.

\subsection{Examples.}\label{sec:examples}
 We now come back to the three natural examples of car arrivals law mentioned in the introduction, namely binary, Poisson and geometric car arrivals. Recall that to check it $ t_c \geq 1$, it suffices to verify if $ (1- G'(1))^2 \geq 2 G''(1)$.  \medskip \\
 \textit{Binary arrivals.} Let us start by the simplest example of non trivial car arrival distribution, namely the binary distribution $ \mu_ \alpha$: each vertex receives either two cars with probability $ \alpha/2$ or no car with probability $ 1 - \alpha/2$ independently for some parameter $ \alpha \in [0,1]$.  In this case, the generating function of  $ \mu_ \alpha$ is $ G(t) := (1- \alpha/2) + \alpha/2 t^2$.
 \begin{eqnarray*} t_c \geq 1\ \Leftrightarrow \ (1- G'(1))^2 \geq 2 G''(1) 
\ \Leftrightarrow\  (1- \alpha)^2 \geq 2 \alpha 
\  \Leftrightarrow\  \alpha \notin (2 - \sqrt{3}, 2 + \sqrt{3})
 \end{eqnarray*}
Moreover, the equation $ (G(t)- tG'(t))^2 =  2 G(t)G''(t)$ is a quadratic equation in the variable $t^2$ and has only one solution of $t_c =  \sqrt{\frac{2}{ \sqrt{3}- 1} }\sqrt{\frac{2 - \alpha}{\alpha}}$ which is larger than 1 when $ \alpha \leq 2 - \sqrt{3}$. 
 At this $t_c$ and recalling that $ \varphi(y) = (y+1) G(y) - y (y-1) G'(y)$,  we obtain 
 \begin{eqnarray*} \frac{t_c G(t_c)}{ \varphi (t_c)^2} = \frac{3 \sqrt{3 + 2 \sqrt{3}}}{2 \left(3+2 \sqrt{2 \sqrt{3}-3} \sqrt{ \frac{2 - \alpha}{ \alpha}}\right)^2 \sqrt{ \alpha (2- \alpha)}} \leq \frac{1}{4}\quad\mbox{ when }\alpha \leq 2 - \sqrt{3}
 \end{eqnarray*}
Thus, using Theorem \ref{thm:locthereshold}, we obtain that the parking process is subcritical if and only if 
$$  \alpha \leq 2 - \sqrt{3} \quad \mbox{ and } \quad q \leq \frac{1}{2} \left( 1 + \sqrt{1 - \frac{6 \sqrt{2 \sqrt{3}+3}}{ \sqrt{ \alpha (2- \alpha)} \left( 3 + \sqrt{ 2 \sqrt{3}- 3} \sqrt{ \frac{2- \alpha}{ \alpha}}\right)^2}}\right). $$\medskip \\
\textit{Geometric arrivals.} Now assume that the car arrival distribution$ \mu_ \alpha = \sum_{k \geq 0} p^k (1-p) \delta_k$ is a geometric distribution with mean $ \alpha = p/(1-p)$. Its generating function is then $G(t) :=1/(1+ \alpha - \alpha t)$. Thus 
 \begin{eqnarray*} t_c \geq 1\ \Leftrightarrow \ (1- G'(1))^2 \geq 2 G''(1) 
\ \Leftrightarrow\  (1- \alpha)^2 \geq 4 \alpha 
\  \Leftrightarrow\  \alpha \notin (1/3,1)
 \end{eqnarray*}
Moreover, the quantity $ (G(t)- tG'(t))^2 -  2 G(t)G''(t)$ is a rational fraction in $t$ and its numerator is a linear function of $t$. We can thus determine, for $ \alpha \leq 1/3$, 
$$t_c = \frac{1 + \alpha}{4 \alpha}, $$
and we obtain
 \begin{eqnarray*} \frac{t_c G(t_c)}{ \varphi (t_c)^2} = \frac{27 \alpha ( 1+ \alpha)^2}{4 (1+ 9 \alpha)^2}.
 \end{eqnarray*}
Using Theorem \ref{thm:locthereshold}, the parking process is subcritical if and only if  
$$  \alpha \leq \frac{1}{3}  \quad \mbox{ and } \quad  q \leq \frac{1}{2} \left( 1 + \frac{(1-3 \alpha )^{3/2}}{1 + 9 \alpha}\right). $$ \medskip \\
\textit{Poisson arrivals.} Lastly, we assume that the car arrivals distribution  $ \mu_ \alpha$ is Poisson with mean $ \alpha$, which means that its generating function is $ G(t) =  \exp\left( a (t-1)\right)$. Then 
 \begin{eqnarray*} t_c \geq 1\ \Leftrightarrow \ (1- G'(1))^2 \geq 2 G''(1) 
\ \Leftrightarrow\  (1- \alpha)^2 \geq 2 \alpha^2 
\  \Leftrightarrow\  \alpha \in [-\sqrt{2}-1,\sqrt{2}-1]
 \end{eqnarray*}
Moreover, the quantity $ (G(t)- tG'(t))^2 -  2 G(t)G''(t)$ is a quadratic polynomial in $ t$ multiplied by  $ \mathrm{e}^{2 a (t-1)}$. It cancels twice, one of the two solutions is positive and corresponds to $$t_c = \frac{ \sqrt{2}- 1}{ \alpha}.$$
At this $t_c$, we obtain 
$$ \frac{t_c G(t_c)}{ \varphi (t_c)^2}  = \frac{( \sqrt{2}-1) \alpha \mathrm{e}^{ \alpha - (\sqrt{2}-1)}}{ (\sqrt{ 2} a - 4 + 3 \sqrt{2})^2}.$$
and thanks to Theorem \ref{thm:locthereshold}, the parking process is subcritical if and only if

$$ \alpha \leq \sqrt{2} - 1  \quad \mbox{ and } q \leq \frac{1}{2} \left( 1+ \sqrt{ 1 - \frac{ 2( \sqrt{2}- 1)  \alpha \mathrm{e}^{ \alpha -( \sqrt{2}- 1)}}{ ( \alpha + 3  - 2 \sqrt{2})^2}}\right).$$

\begin{figure}[!h]
 \begin{center}
 \includegraphics[width=5.5cm]{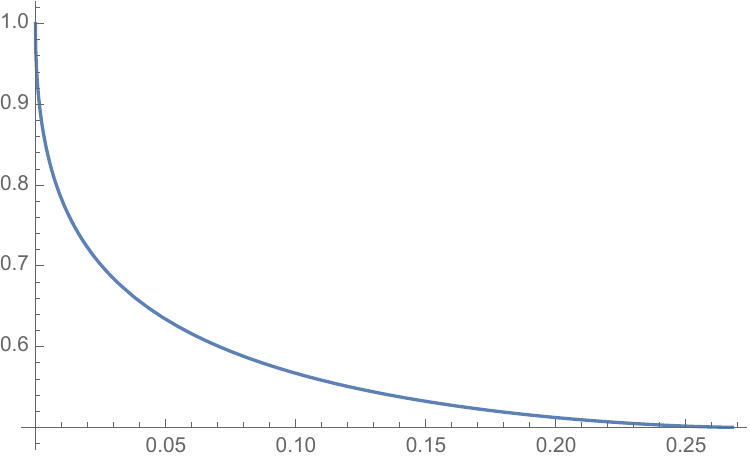}  \includegraphics[width=5.5cm]{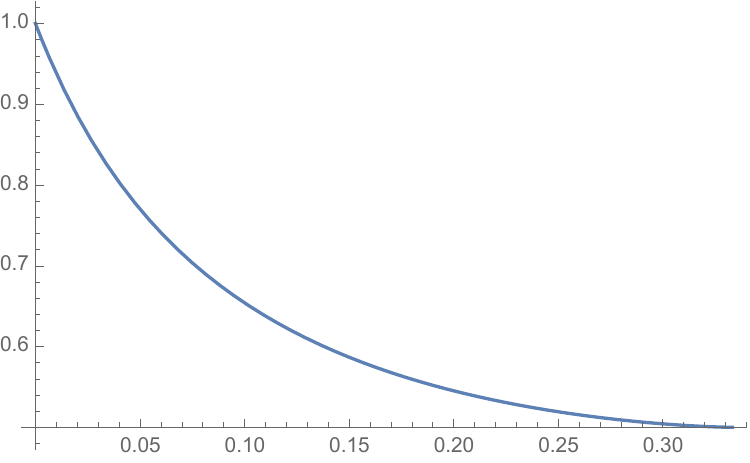}\includegraphics[width=5.5cm]{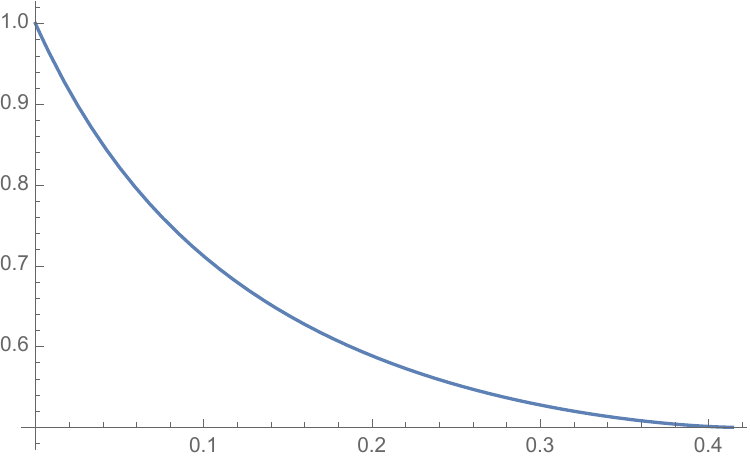}
 \caption{From left to right, the critical value for $q$ as a function of $ \alpha$ for car arrivals with respectively binary, geometric and Poisson distribution. Note that for the geometric 	and Poisson distribution, this function has a finite derivative as $ \alpha$ goes to $0$, whereas it is of the form $ 1- \mathrm{cst} \sqrt{a}$ for some $ \mathrm{cst}>0$ in the binary case}
 \end{center}
 \end{figure}
 \medskip 
\noindent\textit{Stable cases.} A last important example is when the car arrivals distribution is non-generic. More precisely, we assume here that there exists a polynomial $P$, three reals $C < 0$,   $ \rho > 1$ and $\alpha \in (2,3)$ such that 
$$ G(t) = P(t) + C \left(1- \frac{t}{ \rho} \right)^ \alpha,$$
and such that the parameter $ t_c$ in Theorem \ref{thm:locthereshold} is equal to $ \rho$.
Assume that $G$ is the generating function of a critical law, in the sense that the inequality on the right in Theorem \ref{thm:locthereshold} is an equality. 
Then, by Lemma \ref{lem:xboltzmann}, the generating function of $X$ the number of cars visiting the root is proportional to the function $ y \mapsto F(x_c, y)$ where $x_c$ is $ \hat{x} (t_c)$. Using the asymptotics of Chen \cite[Theorem 3]{chen2021enumeration}, this function has a radius of convergence $ t_c$. In particular, the tail of the law of $X$ decays exponentially fast. However, this is not the case for the law of the size $ | \mathcal{C} ( \varnothing)|$ of the cluster of the root. For example, by Lemma \ref{lem:xboltzmann} and conditionally on $X=1$, the generating function of the size of the cluster of the root is $ x \mapsto F ( x_c x, 0). $ In particular, using again the asymptotic provided by Chen in  \cite[Theorem 3]{chen2021enumeration}, we have the coefficient in front of $ x^n$ in $F ( x_c x, 0)$ is asymptotically of order
$$ [x^n]F ( x_c x, 0) \underset{n \to \infty}{\sim} C \cdot x_c^n \cdot x_c^{-n} n^{\frac{- 2 \alpha +1}{ \alpha-1}} = C \cdot n^{\frac{- 2 \alpha +1}{ \alpha-1}}, $$
for some constant $C> 0$ that depends on $ \alpha$. Note that this  exponent  $(- 2 \alpha +1)/ ( \alpha-1)$ is increasing from  $ -3$ to $ -5/2$ for $ \alpha \in (2,3)$.
\subsection{Perspectives}
In this work, we only consider the parking process on supercritical Bienaym\'e--Galton--Watson trees with a geometric offspring distribution. However, we believe that our technique can be extended to more general trees. The challenge in more general cases will be to enumerate the fully parked trees and include in this enumeration the free spots in the underlying tree. We will thus at least need an additional catalytic variable counting the number of adjacent vertices of the fully parked tree inside the underlying tree. We can also imagine to include in the enumeration the possibility for a vertex to carry more the one car. 
Moreover, we may make the car arrival law depend on the degree of the vertex (in the underlying tree). For example, by allowing car arrivals only on the leaves, this model is then very close to the Derrida-Retaux model \cite{chen2021derrida}. 
\bibliographystyle{siam}
\bibliography{/Users/contat/Dropbox/Articles/biblio.bib}
\end{document}